\documentclass[12pt]{amsart}

\usepackage{amsthm}
\usepackage{mathscinet}
\usepackage{amssymb}
\usepackage{graphicx}									 
\usepackage{stmaryrd}                  
\usepackage{footmisc}                  
\usepackage{paralist}
\usepackage{wrapfig}
\usepackage{hyperref}
\usepackage[draft]{pdfcomment}

\usepackage[ngerman,british]{babel}
\usepackage[arrow, matrix, curve]{xy}
\usepackage{color}

\usepackage[T1]{fontenc}
\usepackage[utf8]{inputenc}

\usepackage{enumitem}
\usepackage{fancyvrb}

\newtheorem*{thm-plain}{Theorem}

\newtheorem{theorem}{Theorem}[section]

\newtheorem{lemma}[theorem]{Lemma}
\newtheorem{corollary}[theorem]{Corollary}
\newtheorem{proposition}[theorem]{Proposition}
\newtheorem{definition}[theorem]{Definition}

\theoremstyle{remark}

\newtheorem*{question}{Question}
\newtheorem{remark}{Remark}[section]
\newtheorem{example}{Example}[section]

\newcommand{\R}{\mathbb{R}}
\newcommand{\N}{\mathbb{N}}

\newcommand{\Alex}{\mathrm{Alex}}
\newcommand{\Aff}{\mathcal{A}}

\newcommand{\To}{\rightarrow}

\newcommand{\MTo}{\mapsto}

\newcommand{\diam}{\operatorname{diam}}
\newcommand{\im}{\operatorname{im}}

\newcommand{\interior}{\operatorname{interior}}

\title[Affine functions on Alexandrov spaces]{Affine functions on Alexandrov spaces}
\author{Christian Lange}
\author{Stephan Stadler}
\address{Christian Lange, Mathematisches Institut der Universit\"at zu K\"oln, Weyertal 86-90, 50931 K\"oln, Germany}
\email{clange@math.uni-koeln.de}
\address{Stephan Stadler, Mathematisches Institut der Universität München, Theresienstr. 39, 80333 München, Germany}
\email{stephan.stadler@math.lmu.de}


\begin{document}

\begin{abstract} We show that every finite-dimensional Alexandrov space X with curvature bounded from below embeds canonically into a product of an Alexandrov space with the same curvature bound 
and a Euclidean space such that each affine function on X comes from an affine function on the Euclidean space.
\end{abstract}
\maketitle	

\section{Introduction}

According to a classical result of Toponogov \cite{T64}, a complete Riemannian manifold $M$ with nonnegative sectional 
curvature that contains a straight line is isometric to the metric product of a nonnegatively curved manifold and a line (see also \cite{C36} for earlier results by Cohn-Vossen in the case of surfaces). In this case the Busemann function 
associated with the straight line is an \emph{affine function}, that is, its restriction to any unit-speed geodesic is 
affine. For nonnegatively curved Alexandrov spaces the same result was obtained by Milka \cite{Mi67}. In general, a map between geodesic metric spaces
is called affine, if it sends unit-speed geodesic segments to geodesics of constant speed. In the special case of $Y=\R$ this definition reads as follows.

\begin{definition} A function $f:X\To \R$ on a geodesic metric space $X$ is affine, if the restriction $f\circ \gamma$ to each geodesic $\gamma : [a,b] \To X$ is affine, i.e. it satisfies $(f\circ \gamma)''=0$.
\end{definition}

The easiest example of an affine function is the projection onto a Euclidean factor. Under some assumptions it is known that a space $X$ with one-sided curvature bound splits as a product $X = Y \times \R$ if it admits a non-constant affine function \cite{AB05}  (see also \cite{I82, Mashiko,Mashiko2} for earlier results). The decisive assumption in 
\cite{AB05} is that the space $X$ is geodesically
 complete in the case of an upper curvature bound or does not have boundary in the case of a lower curvature bound. The example of a Euclidean ball shows that such a splitting cannot exist without this assumption.
In this case the best one can expect is an isometric embedding of $X$ into the product of some space with a real line \cite{LS07}. For an upper curvature bound the existence of such an 
embedding is established in \cite{LS07}. In this paper we treat the case of a lower curvature bound. First, we prove 
the following regularity result.

\begin{theorem}\label{thm:Lip_cont} A measurable affine function on a finite-dimensional Alexandrov space with curvature bounded from below is Lipschitz continuous.
\end{theorem}

The measurability assumption could be dropped if one of the following open questions due to Lytchak admitted a positive solution (cf. Section \ref{sub:continuity}).

\begin{question} Is a dense convex subset of an $n$-dimensional Alexandrov space with curvature bounded from below measurable with respect to the $n$-dimensional Hausdorff measure? Does its complement have measure zero?
\end{question}

In dimension two the answers are yes (see also \cite{Mashiko} for a proof of Lipschitz continuity of affine functions on Alexandrov surfaces without boundary).

In the following we restrict ourselves 
to Lipschitz continuous affine functions. As in \cite{LS07} we choose a slightly more general formulation that takes 
into account all affine functions at once. We denote by $p_Y$ and $p_H$ the natural projections from 
the product $Y \times \R^m$ onto the factors $Y$ and $\R^m$, respectively.

\begin{theorem}\label{thm:char_affine} Let $X \in \Alex^n(\kappa)$, that is an $n$-dimensional Alexandrov space with curvature bounded from below by $\kappa$. Then there exists a geodesic metric space $Y$ and an isometric embedding $i:X\To Y\times \R^{m}$, $m\leq n$, with the following properties: 
\begin{enumerate}[label=(\alph*)]
\item The projection $\pi =p_Y \circ i : X \To Y$, is surjective.
\item Each Lipschitz continuous affine function $f:X\To \R$ factors as $f=\hat{f}\circ p_H \circ i$ where $\hat{f}:\R^{m}\To \R$ is affine. In particular, each Lipschitz continuous affine function on $Y$ is constant.
\item Each isometry of $X$ extends uniquely to an isometry of $Y\times \R^{m}$.
\item The embedding $i$ is open on an open neighborhood of the regular points in $X$.
\end{enumerate}
Moreover, the listed properties uniquely determine the dimension $m$, the space $Y$ up to isometry and the embedding $i:X\To Y\times \R^{m}$ up to composition with isometries.
\end{theorem}

In particular, if $X \in \Alex^n(\kappa)$ admits a non-constant measurable affine function, then $\kappa\leq 0$.
Compared to \cite{LS07} we additionally obtain that the embedding $i$ is open on an open and dense subset of $X$ and that the data $i:X\To Y\times \R^{m}$ are uniquely determined by $X$ and the properties stated in the theorem.

\begin{remark} Given Theorem \ref{thm:curvature_bound} below and a domain invariance theorem for Alexandrov spaces (cf. \cite{BIP10}), property $(d)$ in Theorem \ref{thm:char_affine} actually holds on the set of all interior points of $X$.
\end{remark}

The fibers of the projection $P_H= p_H \circ i: X \To \R^m$ consist of points that
 cannot be separated by Lipschitz continuous affine functions. Each of these fibers is a convex subset of $X$ and projects isometrically into 
$Y$. The whole space $Y$ is made up of such pieces. As in \cite{LS07} this picture suggests that $Y$ has the same curvature bound as $X$. 
Indeed, using a recent result of Petrunin \cite{Pe15} we prove the following statement.

\begin{theorem}\label{thm:curvature_bound} Let $X \in \Alex^n(\kappa)$. Then the completion of the space  $Y$ constructed 
in Theorem \ref{thm:char_affine} has the same curvature bound as $X$. More precisely, $\bar Y \in \Alex^{n-m}(\kappa)$ with 
$m$ being given as in Theorem \ref{thm:char_affine}.
\end{theorem}

The proof of Theorem \ref{thm:char_affine} is based on the same idea as the proof of the analogous result in \cite{LS07} 
in the case of an upper curvature bound. If the conclusion of Theorem \ref{thm:char_affine} holds, then the metric on $Y$ must satisfy $ d(\pi(y),\pi(z))= \sqrt{d(y,z)^2-||F(y)-F(z)||^2}$ 
for all $y,z \in X$ where $F= p_H \circ i:X \To \R^m$. In particular, the right hand side of this expression must be a 
pseudometric on $X$ with $Y$ being the corresponding metric space. To prove \ref{thm:char_affine} we first define some 
Hilbert space $H$ and a natural map $F:X \To H$ that would coincide with $p_H \circ i$ if the theorem were true. Then 
we show that the term $\sqrt{d(y,z)^2-||F(y)-F(z)||^2}$ defines a pseudometric on $X$. These steps work similarly as in 
\cite{LS07} but require different ingredients. For instance, our proof relies on the theory of quasigeodesics \cite{QG95, Pe07}. The dimension of $H$ then turns out to be finite (cf. Lemma \ref{lem:affine_hilbert}). The proof 
of the openess statement in Theorem \ref{thm:char_affine}, $(d)$, is based on the fact that finite-dimensional 
Alexandrov spaces are locally Euclidean on an open and dense subset and uses the domain invariance theorem. 
This property, together with Petrunin's result \cite{Pe15}, is needed in the proof of Theorem \ref{thm:curvature_bound}. The listed ingredients restrict our proof to finite dimensions.
\begin{question} Do the results of this paper continue to hold for infinite-dimensional Alexandrov spaces with curvature bounded from below?
\end{question}

\section{Preliminaries}

\subsection{Spaces}\label{sub:spaces}

By $d$ we denote the distance in metric spaces without an extra reference to the space. By $B_r(x)$ we denote
 the open metric ball of radius $r$ around a point $x$. A \emph{pseudometric} is a metric for which the distance 
between different points may be zero. Identifying points with pseudo distance equal to zero yields a metric space. 

A \emph{geodesic} in a metric space is a length minimizing curve parametrized proportionally to arclength. A 
metric space is \emph{geodesic} if each pair of points is connected by a geodesic. A subspace of a geodesic 
space is called \emph{(totally) convex} if it contains every geodesic between pairs of its points. An Alexandrov space 
with curvature $\geq \kappa$ is a complete, geodesic metric space in which triangles are not thinner than in 
the two-dimensional model space of constant curvature $\kappa$. The Hausdorff 
dimension of such a space is an integer or infinite. We denote the set of $n$-dimensional Alexandrov spaces with curvature $\geq \kappa$ 
by $\Alex^n(\kappa)$ and consider only finite-dimensional spaces. We refer to \cite{MR1835418,BGP92} for a more detailed discussion of such spaces. We will 
need the following estimate that is a direct consequence of the $\Alex(\kappa)$ definition and hyperbolic trigonometry.
 There exist numbers $A=A(\kappa)$, $r=r(\kappa)$ such that for each triple $x_1,x_2,x_3$ in a space $X \in \Alex(\kappa)$ 
with $d(x_i,x_j)\leq r$ the following holds. Let $m_i$ be the midpoint between $x_j$ and $x_k$ ($j \neq i \neq k \neq j$). 
\begin{itemize}
\item If $d(x_1,x_2),d(x_1,x_3)\leq 2t$, then $d(m_2,m_3)\geq\frac{1}{2}d(x_2,x_3)(1-At^2)$.
\end{itemize}
Note that for $\kappa \geq 0$ we may take $A=0$.

\subsection{Space of directions and tangent cone}\label{sub:tangent_cones} For two points $x$ and $y$ in an Alexandrov space $X$ we
 denote a geodesic between $x$ and $y$ by $[xy]$ and the interior of this geodesic by $(xy)$. Such a geodesic always exists, but is in general not uniquely determined by $x$ and $y$. We put
\[
	\Sigma_x' := \{[xy] | y \in X \backslash \{y\} \}/ \sim.
\]
where the equivalence relation is defined such that $[xy] \sim [xz]$ if and only if $[xy] \subset [xz]$ 
or $[xz] \subset [xy]$. Note in this regard that geodesics in $X$ cannot branch. Measurement of angles defines a metric on $\Sigma_x'$ \cite[§4.3]{MR1835418}. The \emph{space of directions} $\Sigma_x$ is defined to be the metric completion
 of $\Sigma_x'$. The metric cone over the space of directions is called the \emph{tangent cone} of $X$ 
at $x$ and is denoted by $T_x$ or $T_x X$. Alternatively, the tangent cone $T_x$ can be defined as pointed Gromov-Hausdorff limit of rescaled 
versions of $X$ \cite[Theorem 7.8.1]{BGP92}. If  $X \in \Alex^{n}(\kappa)$, then $\Sigma_x \in \Alex^{n-1}(1)$ 
and $T_x \in \Alex^{n}(0)$. Due to a result of Perelman, every point in a finite-dimensional Alexandrov space has a neighborhood that is pointed homeomorphic to the tangent cone at that point (cf. \cite[Thm.~10.10.2]{MR1835418}). A point in $X$ is called \emph{regular}, if its tangent cone is isometric to $\R^n$. 
The set of regular points in $X$ is convex \cite[Corollary~1.10]{Pe98} and dense in $X$ \cite[Corollary~10.9.13]{MR1835418}. 
The boundary of an Alexandrov space can be defined inductively via the spaces of directions. A point belongs 
to the boundary if and only if the boundary of its space of directions is non-empty. The \emph{interior} is the complement of the boundary. The interior of an 
$\Alex^n(\kappa)$ space $X$ is open in $X$ \cite[13.3. b)]{BGP92}. A geodesic $\gamma:[a,b) \To X$ that starts in the interior of $X$ stays
 in the interior of $X$ \cite[Thm.~1.1 A]{Pe98}. In particular, the interior of $X$ is convex.

\subsection{$\lambda$-concave functions and quasigeodesics}\label{sub:qg} Let $X \in \Alex^n(\kappa)$ 
and $U$ be an open subset \emph{in the interior} of $X$. 
A continuous function $f$ on $U$ is called \emph{$\lambda$-concave}, if for all unit-speed
 geodesics $\gamma$ in $U$ the function 
\[ 
	f\circ \gamma(t)-\frac{\lambda}{2} t^2
\]
is concave (cf. \cite[Cor.~3.3.2]{Pe15}). A curve $\gamma$ in $U$ is called \emph{quasigeodesic}, if for any $\lambda \in \R$ and any $\lambda$-concave function
 $f$, the composition $f \circ \gamma$ is $\lambda$-concave. Quasigeodesics have nice properties, e.g. they are unit 
speed curves \cite[Thm.~7.3.3]{Pe07} and for any point $x \in X$ and any direction $\xi \in \Sigma_x$ there exists a 
quasigeodesic with $\gamma(0)=x$, $\gamma^+(0)=\xi$. Here $\gamma^+(0)$ is defined to be the limit in $\Sigma_x$ of 
the directions $[x\gamma(t)]$ for $t\searrow 0$ (cf. \cite[Thm.~A.0.1]{Pe07}). Observe that if $f: X \To \R$ is 
affine and continuous in the interior of $X$, then its restriction to a quasigeodesic in the 
interior of $X$ is affine.

\subsection{Space of affine functions}\label{sub:functions} A map $f:X \To Y$ is called \emph{$L$-Lipschitz} if
 it satisfies $d(f(x),f(z))\leq Ld(x,z)$ for all $x,z \in X$. The smallest $L$ as above is called the 
\emph{optimal Lipschitz constant} of $f$. Let $X$ be a geodesic metric space. As in \cite{LS07} we denote by $\tilde\Aff(X)$ the
 vector space of Lipschitz continuous affine functions on $X$ and by $\Aff(X)=\tilde \Aff(X) / \mathrm{Const}(X)$ the quotient vector 
space by the subspace of constant functions. For $f \in \tilde \Aff(X)$ we denote by $[f]$ the corresponding element 
in $\Aff(X)$. The optimal Lipschitz constant defines a norm on $\Aff(X)$ with respect to which $\Aff(X)$ is a Banach space.

The evaluation map $E: X \times X \To \Aff^*$ is defined by $E(x,y)([f])=f(y)-f(x)$ and satisfies $||E(x,y)||\leq d(x,y)$. 
It maps geodesics to linear intervals of the Banach space $\Aff(X)^*$. Observe that $E(x,y)=0$ if and only if the points 
$x$ and $y$ cannot be separated by a Lipschitz continuous affine function on $X$. By $E_x: X \To \Aff^*$ is denoted the
 restriction $E_x(y)([f])=f(y)-f(x)$. The following lemma holds. 

\begin{lemma}[{\cite[3.1]{LS07}}] \label{lem:evaluation}
The evaluation map $E_x:X \To \Aff^*$ is $1$-Lipschitz. For each Lipschitz continuous affine function 
$f\in \tilde \Aff$ we have $[E_x(\cdot)([f])]=[f]$.
\end{lemma}

\subsection{Measures}

The (n-dimensional) Hausdorff measure on $X\in \Alex^n(\kappa)$ will be denoted by $\mu_n$. 
Recall that for a subset $E\subset X$ we have 
$$
\mu_n(E):=\lim_{\delta\to 0}\inf\{\sum_{i=1}^\infty (\diam A_i)^n :\ E\subset\bigcup_{i=1}^\infty A_i,\ \diam A_i<\delta\}
$$
and that $\mu_n$ is a Borel measure on $X$ (by the same proof as in the Euclidean case, cf. \cite[Sec.~2.1, Thm.~1, Claim \#3]{MR1158660}). 
Since the diameter of a subset of $X$ does not increase when passing to its closure, we can restrict
to closed coverings in the above definition. It follows that for every set $A\subset X$ there exists a Borel 
set $B$ such that $A\subset B$ and $\mu_n(A)=\mu_n(B)$. From the Bishop inequality (\cite[Thm.~10.6.8]{MR1835418}) we get that $\mu_n$ takes finite values on compact subsets of $X$.
So $\mu_n$ is a Radon measure in the sense of \cite{MR1158660}.
For a 
subset $A\subset X$ and a point $x\in X$ we define the density of $A$ at the point $x$ by
$$
\Theta_A(x):=\lim_{r\to 0}\frac{\mu_n(B_r(x)\cap A)}{\omega_n r^n}
$$
if the limit exists. (Here $\omega_n$ denotes the volume of the n-dimensional Euclidean unit ball.)
The Lebesgue density theorem says that for a $\mu_n$-measurable set $A\subset \R^n$ the density $\Theta_A(x)$ 
is equal to one
for almost every point $x$ in $A$ while it is equal to zero for almost every point in the complement $X-A$
 (see \cite[Ch.~1.7, Cor.~3]{MR1158660}). Since there exists a subset $S\subset X$ of singular points such that $X-S$ 
 is locally bi-Lipschitz equivalent to $\R^n$ and $\mu_n(S)=0$ (see \cite{OS94}),
 the statement remains true for $\mu_n$-measurable sets $A\subset X$.

\section{Continuity and Lipschitz continuity}

\subsection{Measuring convex dense sets}
\label{sub:continuity}

\begin{lemma}\label{lem:full}
Let $C$ be a dense convex subset of $X$. If there exists a point $p\in X$ and a radius $r_0>0$,
such that $\mu_n(C\cap B_{r_0}(p))>0$, then $\mu_n(C\cap B_{\rho}(q))=\mu_n(B_{\rho}(q))$ holds for every 
$q\in X$ and every $\rho>0$.
\end{lemma}

\begin{proof}
Without loss of generality we may assume that $r_0=1$. We first treat the case $\kappa\geq 0$.
Set $M:=\mu_n(B_1(p)\cap C)$. For some $x\in X$ and $r>0$ choose $y\in B_{\frac{r}{2}}(x)\cap C$. To find measure in $B_r(x)\cap C$ 
we will pull $B_1(p)\cap C$ towards $y$ by a co-Lipschitz map. Set $\lambda:=\frac{r}{2(r+d(p,x)+1)}$
and choose for every point $p'\in B_1(p)\cap C$ a shortest geodesic $[p'y]$ and define $y'$ as the unique point on 
$[p'y]$ with $d(y,y')=\lambda d(y,p')$. Then the assignment $p'\mapsto y'$ is $\lambda$-co-Lipschitz and since 
$d(y,y')=d(y,p')\lambda\leq (r+d(p,x)+1)\lambda<r/2$, it maps $B_1(p)$ into $B_r(x)$. Hence, for $r<1$ we have
 \begin{equation}\label{equ:coLip}
  \mu_n(B_r(x)\cap C)\geq \lambda^n M \geq \underbrace{\frac{M}{(2d(p,x)+4)^n}}_{=:M_x}r^n.
 \end{equation}
In particular, $\Theta_C(x)\geq M_x/\omega_n >0$. 

Now pick some $q\in X$ and $\rho>0$. By Borel regularity, there is a measurable subset $\tilde C\subset B_\rho(q)$
with $C\cap B_\rho(q)\subset \tilde C$ and $\mu_n(C\cap B_\rho(q))=\mu_n(\tilde C)$. From above we obtain
$\Theta_{\tilde C}(x)>0$ for all $x\in B_\rho(q)$. Thus the Lebesgue density theorem implies 
$\mu_n(\tilde C)=\mu_n(B_\rho(q))$. This finishes the proof for $\kappa\geq 0$. If $\kappa<0$ we have to replace
Equation \ref{equ:coLip} by 
$$
\mu_n(B_r(x)\cap C)\geq c^n M
$$
where $c=\frac{\sinh(-\kappa\lambda R)}{\sinh(-\kappa R)}$ (see \cite[p.~370]{MR1835418}) and $R$ is some fixed
 radius $\geq d(p,x)+1$.
Since $\lim_{r\to 0}\frac{c}{r}>0$ the argument goes through as before.
\end{proof}

\subsection{Continuity}
In this subsection $f:X\to \R$ will denote a nonconstant affine function on an Alexandrov space $X \in \Alex^n(\kappa)$
 that is measurable with respect to the $n$-dimensional Hausdorff measure of $X$. For a value $t\in \im(f)$ we define
 $F_t:=f^{-1}(t)$ to be the {\em fiber over $t$}.

\begin{remark}
 If $f:X\to \R$ is an affine function, then the inverse image of every convex set is convex.
\end{remark}

\begin{lemma}\label{lem:fibernull}
Every fiber $F_t$ is a $\mu_n$-null set.
\end{lemma}

\begin{proof}
 Assume there is a radius $r>0$ and a point $p\in F_t$ such that $\mu_n(F_t\cap B_r(p))>0$. 
Choose a point $q$ with $f(q)=t'\neq t$. We may assume that $t<t'$.
 Set $R:=r+d(p,q)$. For every $\tau\in(0,1)$ we have a map $\varphi_\tau:B_R(q)\to B_{\tau R}(q)$ 
which sends a point $z\in B_R(q)$ to a point $\varphi_\tau(z)$
 on a shortest geodesic from $q$ to $z$ which lies at distance $\tau d(q,z)$ from $q$. Note that 
$\varphi_\tau$ is co-Lipschitz and maps $F_t\cap B_R(q)$ to 
 $F_{(1-\tau)t+\tau t'}\cap B_{\tau R}(q)$. Hence $\mu_n(F_{\theta}\cap B_{R}(q))>0$ for
every $\theta\in(t,t')$. Since the sets $F_{\theta}\cap B_{R}(q)$ are all measurable and disjoint, 
this contradicts $\mu_n(B_R(q))<\infty$.
\end{proof}

\begin{lemma}\label{lem:nodense}
 No fiber of $f$ is dense in $X$.
\end{lemma}

\begin{proof}
  Assume that there is a fiber $F_{t_0}$ which is dense in $X$. Then, by Lemma \ref{lem:full} and Lemma \ref{lem:fibernull} there is 
  a unique minimal closed interval $[a,b]$ containing $t_0$ such that $f^{-1}([a,b])$ has full measure in $X$. By minimality and Lemma \ref{lem:full}
  we conclude that the inverse image of any closed subinterval of $[a,b]$ which contains one of the endpoints but not the value $t_0$ has full measure in 
  $X$. This contradicts Lemma \ref{lem:fibernull} because it implies that one of the fibers $F_a$ or $F_b$ has full measure in $X$.
\end{proof}

\begin{corollary}\label{cor:clofib}
Every fiber $F_t$ is closed.
\end{corollary}

\begin{proof}
If $F_{t_0}$ is not closed for some $t_0\in\im f$, then the restriction $f|_{\bar F_{t_0}}$ is a nonconstant 
affine function on the closure $\bar F_{t_0}$, which is itself an Alexandrov space of dimension $k\leq n$. 
This contradicts Lemma \ref{lem:nodense} because $f|_{\bar F_{t_0}}$ has a dense fiber.
\end{proof} 

\begin{corollary}\label{cor:affine_continuity}
The map $f$ is continuous.
\end{corollary}

\begin{proof}
We will show that $f^{-1}([a,b])$ is closed for every compact interval $[a,b]$ in $\R$. Assume that the closure 
of $f^{-1}([a,b])$ contains a point $x$ with $f(x)\notin [a,b]$. We may assume $f(x)>b$. Choose a sequence
$(x_i)$ in $f^{-1}([a,b])$ with $x_i\to x$. Then the geodesic segments connecting $x_i$ to $x$ intersect the fiber
$F_b$ in points $y_i$. But then $y_i\to x$ and therefore $x\in F_b$ by Corollary \ref{cor:clofib}.
\end{proof}

\begin{corollary}[Local Lipschitz regularity in the interior]\label{lem:loLip}
The function $f$ is locally Lipschitz continuous in the interior of $X$.
\end{corollary}

\begin{proof}
This follows from Corollary 3.3.2 in \cite{Pe07}.
\end{proof}

\subsection{Lipschitz continuity} \label{sub:lipschitz_continuity}

Let $X$ be an $\Alex^n(\kappa)$ space and let $f:X \To \R$ be a non-constant locally Lipschitz continuous affine function.
 Recall that the tangent cone $T_x$ at a point $x \in X$ is an $\Alex^n(0)$ space. The function $f$ has a well-defined 
directional derivative $d_xf: T_x \To \R$ that is itself
 a Lipschitz continuous affine function and obeys the same Lipschitz constant as $f$, that is 
$||d_x f||\leq||f||$. If $v \in \Sigma_x$ is represented by a curve $\gamma:[0,\epsilon)\to X$ 
with $\gamma(0)=x$ (cf. \cite[2.1]{Pe07}), then $d_x f(v)$ is given by $(f\circ \gamma)'(0)$. If $d_x f$ 
attains positive values, then it attains its maximum on $\Sigma_x$ at a unique unit-vector $\xi_x$ \cite{Pe07}. 
In this case the vector $d_x f(\xi_x) \xi_x$ is called the \emph{gradient of $f$ at $x$} and is denoted by $\nabla_x f$. 
Otherwise, one sets $\nabla_x f=o_x \in T_x$. The length $|\nabla_x f|$ is called the 
\emph{absolute gradient} of $f$ at $x$. We have $|\nabla_x f|= \max \{0, \sup_{x\neq z} \left(f(z)-f(x)\right)/d(x,z) \}$. 
This implies that $|\nabla_x f|$ is lower semi-continuous, i.e. $\lim \inf  |\nabla_{x_i} f| \geq |\nabla_{x} f|$ for 
$x_i \To x$  \cite[Lem. 3.2.1]{QG95} (cf. \cite{Plaut}).
By Corollary \ref{lem:loLip} $|\nabla_x f|$ is finite for points $x$ in the 
interior of $X$.

In the following we frequently use that the restriction of $f$ 
to quasigeodesics in the interior is affine (cf. Section \ref{sub:qg}).

\begin{lemma} Let $\gamma: [a,b) \To X$ be a quasigeodesic such that $\gamma((a,b))$ is contained in the interior
of $X$. Then $p(t)=|\nabla_{\gamma(t)} f|$ is constant $p_0$ on $(a,b)$ and we have $p(a)\leq p_0$.
\end{lemma}
\begin{proof} Since the restriction of $f$ to $\gamma$ is continuous, we have $\mathrm{lim\ inf}_{t\To a} p(t) \geq p(a)$ (cf. Section \ref{sub:functions}). 
It remains to prove that $p$ is constant on $(a,b)$. The statement is local, therefore we may assume that $\gamma$ is
 parameterized by the arclength and has length smaller than $r/2$. Here and below $r$ and $A$ are chosen as in Section \ref{sub:spaces}.

We claim that for each $s \in (a,b)$ and sufficiently small $t$ with $|t|\leq \min \{s-a,b-s\}$ one has
 $p(s)(1-At^2)\leq p(s+t)$. This claim implies that $p$ is locally Lipschitz and that the differential of $p$ vanishes at each point in $(a,b)$.
Hence $p$ has to be constant on $(a,b)$.

In order to prove the claim, choose $t>0$ such that $B_{5t}(\gamma(s))$ is contained in the interior of $X$ 
and pick a point $m \in B_{t}(\gamma(s)) \backslash\{\gamma(s)\}$. We can extend a minimizing geodesic between $\gamma(s-t)$ and $m$ through $m$ to a quasigeodesic
of twice the length which lies completely in the interior of $X$. Let $z \neq \gamma(s-t)$ be its endpoint. 
Again using the fact that restrictions of $f$ to quasigeodesics are affine, we obtain 
$f(m)-f(\gamma(s))=\frac{1}{2}(f(z)-f(\gamma(s+t))$. On the other hand, we have $d(m,\gamma(s)) \geq\frac{1}{2}d(z,\gamma(s+t))(1-At^2)$ 
(see Section \ref{sub:spaces}). Hence,
\[
		\max \left\{0,\frac{f(m)-f(\gamma(s))}{d(m,\gamma(s))} \right\} \leq \max \left\{0,\frac{f(z)-f(\gamma(s+t))}{d(z,\gamma(s+t))(1-At^2)}\right\}.
\]
We deduce that $p(s)(1-At^2)\leq p(s+t)$ and the claim follows.
\end{proof}

As a corollary we obtain Theorem \ref{thm:Lip_cont}.

\begin{corollary}\label{cor:abs_grad} Let $x \in X$ be a point in the interior of $X$ and let $z \in X$ be any point.
 Then $|\nabla_z f|\leq |\nabla_x f|$. In particular, $|\nabla_x f|$ is constant $L_0< \infty$ in the interior of $X$ and $f$ is Lipschitz continuous with optimal Lipschitz constant $L_0$.
\end{corollary}
\begin{proof} Extend a geodesic from $y$ to $x$ through $x$ to a quasigeodesic that contains $x$ in its interior. 
Now the claim follows from the preceding lemma.
\end{proof}

In particular, we have the following.

\begin{corollary} \label{cor:sequence_dense}
For any point in the interior of $X$ we have $|\nabla_x (-f)|= |\nabla_x f|=||f||$.
\end{corollary} 

\begin{remark}
The analogous statement in \cite{LS07} only holds on a convex and dense subset.
\end{remark} 

Let $x$ be a point in the interior of $X$ and consider the unit vectors $v^{\pm}$ in $T_x$ 
with $d_xf (v^{\pm})= \pm |\nabla_x f|$. Since $d_xf$ is $||f||$-Lipschitz, we must have $d(v^+,v^-)=2$, 
that is the concatenation of the homogenous rays $\gamma^{\pm}(t)=t v^{\pm}$ is a line in $T_x$ (cf. \cite{LS07}). By construction $(d_x f \circ \gamma)'=||f||$ and hence $||d_x f||=||f||$. This proves the following statement.

\begin{lemma} \label{lem:isometric_embedding} For any point $x$ in the interior of $X$ the map $\Aff (X) \To \Aff(T_x X)$, $[f] \mapsto [d_x f]$ is an isometric embedding.
\end{lemma}

For a regular point $x$ in $X$, the tangent cone $T_x X$ is isometric to $\R^n$. Since $\Aff(T_x X)$ is a Hilbert space, so is $\Aff (X)$ by the preceding lemma.

\begin{lemma}\label{lem:affine_hilbert} The Banach space $\mathcal{A}(X)$ of Lipschitz continuous 
affine functions on $X$ is a Hilbert space of dimension not larger than $n$.
\end{lemma}

\section{Normalization}

\subsection{Basic splitting results}

\begin{lemma}\label{lem:line_splitting} Let $X \in \Alex^n(0)$ and $f: X \To \R$ be a Lipschitz continuous affine function. 
If for some unit-speed line  $\gamma$ in $X$ we have $\infty > (f\circ \gamma)'=||f||>0$ then $X$ splits as $X=Z\times \R$ 
and $f$ is given by $f(z,t)=||f|| t$.
\end{lemma}
\begin{proof} According to the well-known splitting theorem \cite[Thm.~10.5.1]{MR1835418} the line 
$\gamma$ defines a line factor of $X$, that is we have a splitting $X=Z\times \R$. Let $z\in Z$ be 
arbitrary and let $\gamma_z$ be the line through $(z,0)$ parallel to $\gamma$. Since $f$ is Lipschitz continuous, 
we must have $(f\circ \gamma)'=(f\circ \gamma_z)'$. The condition $\infty > (f\circ \gamma)'=||f||>0$ 
implies that $f$ is constant on level sets $Z\times \{t_0\}$, $t_0 \in \R$. Now the last statement is clear, too.
\end{proof} 

Using the preceding lemma we obtain the following analogue of \cite[Lemma 4.2]{LS07}. The proof is the
 same as in \cite{LS07}.

\begin{lemma}\label{lem:splitting} Let $X$ be an $\Alex^n(0)$ space and $F: X \To \R^k$ be a $1$-Lipschitz 
affine map with coordinates $F_i$. Assume that there is a point $x\in X$ and unit-speed lines $\gamma_1,\ldots,\gamma_k$
 through $x$ such that $(F_i \circ \gamma_i)'=1$. Then $X$ splits as $X=Z \times \R^k$ such that $F$ is
 the projection onto the $\R^n$ factor. 
\end{lemma}

\subsection{Normalized maps and their regular points}

We transfer a property of the evaluation map from \cite{LS07} to the setting of $\Alex^n(\kappa)$ spaces.

\begin{definition} Let $X$ be an $\Alex^n(\kappa)$ space, $H$ be a Hilbert space and $F:X \To H$ an affine map.
 We call $F$ \emph{normalized}, if $F$ is $1$-Lipschitz and for each unit vector $h \in H$ the affine function
 $F^h: X \To \R$ given by $F^h(x)=\left\langle F(x),h \right\rangle$ satisfies $||F^h|| = 1$. 
\end{definition}
An affine function $f: X \To \R$ is normalized if and only if it has norm $1$. 

\begin{example} Let $H_0 \subset H$ be a Hilbert subspace. Then the orthogonal projection $p: H \To H_0$ is 
normalized. If $F: X \To H$ is normalized then so is the composition $p \circ F$.
\end{example}

Using the natural identification between $\Aff(X)$ and its dual $\Aff(X)^*$ we deduce from Lemma \ref{lem:evaluation}:

\begin{lemma} For each point $x\in X$ the evaluation map $E_x: X \To \Aff(X)^*$ is normalized.
\end{lemma}

Given an $L$-Lipschitz continuous affine map $F:X \To H$ to a finite-dimensional Hilbert space $H$ one can 
define directional differentials $d_xF: T_x \To T_{f(x)} = H$ by setting $d_xF(v)= (F\circ \gamma)'$ for a
 quasigeodesic $\gamma$ starting at $x$ in the direction $v$ and extending homogenously. The differentials are again $L$-Lipschitz and affine.

\begin{definition}\label{dfn:regular_normalized} Let $X$ be an $\Alex^n(\kappa)$ space, $H$ be a Hilbert
 space and $F:X \To H$ a normalized affine map. We call a point $x \in X$ regular (with respect to F) if
 $T_x$ has a splitting $T_x= C_x' \times H_x$, with a Hilbert space $H_x$ such that $d_xF: T_x \To H$ is 
a composition of the projection of $T_x$ to $H_x$ and an isometry.
\end{definition}

An affine function $f: X \To \R$ is normalized if and only if its optimal Lipschitz constant is $1$. In 
this case a point $x \in X$ is regular with respect to $f$ if and only if $|\nabla_x f|=|\nabla_x (-f)|=1$ 
as the discussion preceding Lemma \ref{lem:isometric_embedding} together with Lemma \ref{lem:line_splitting} 
shows. We obtain the following analogue of \cite[Lemma 4.8]{LS07}. 

\begin{lemma}\label{lem:normalized_dense} Let $X$ be an $\Alex^n(\kappa)$ space and $F:X \To \R^n$ a 
normalized affine map. Then every point in the interior of $X$ is regular for $F$.
\end{lemma}
\begin{proof} The proof is the same as in \cite{LS07}: Let $f_i$, $i=1\ldots,n$ be the coordinates of
 $F$. If a point $x\in X$ is regular, we must have $|\nabla_x f_i|=|\nabla_x (-f_i)|=1$ for all $i$. 
On the other hand, such a point $x$ is regular, due to Lemma \ref{lem:splitting} and the observations 
preceding Lemma \ref{lem:isometric_embedding}. The statement now follows from Corollary \ref{cor:sequence_dense}.
\end{proof}

\section{Proof of Theorem 1.3}

\subsection{The pseudometric}

The following statement is needed in the proof of Theorem \ref{thm:char_affine}.

\begin{proposition}\label{prp:pseudometric}
Let $X$ be an $\Alex^n(\kappa)$ space, $H$ a Hilbert space and $F:X \To H$ a normalized affine map.
 Then $\tilde d: X \times X \To [0,\infty )$ given by \[\tilde d(y,z)= \sqrt{d(y,z)^2-||F(y)-F(z)||^2}\] 
defines a pseudometric on $X$. \end{proposition}
\begin{proof} 
The proof works in the same way as the proof of \cite{LS07} in the $\mathrm{CAT}(\kappa)$ case, 
since the first variation formula also holds in $\Alex^n(\kappa)$ spaces \cite[Sect.~3.6]{Plaut}.
\end{proof}

\subsection{Isometric embedding} Given an affine normalized map  $F: X \To H$ 
and the associated pseudometric $\tilde d$ on $X$  by Proposition \ref{prp:pseudometric}, let $Y=X/\tilde d$ be the induced metric space. A point in $Y$ is an equivalence class $[x]$ where 
$x \sim x'$ if and only if $\tilde d (x,x')=0$. The map $i: X \To Y \times H$, $x \MTo ([x],F(x))$ is an isometric embedding. 
For a point $x \in X$ we define $Z_x = F^{-1}(F(x))$. 
We record further properties:

\begin{lemma} \label{lem:prop1}Let $F: X \To H$, $\tilde d$ and $Y$ be given as above.
\begin{enumerate}
\item The space $Y$ is geodesic and the projection $\pi: X \To Y$ is affine.
\item For any $x \in X$ the subsets $[x],Z_x\subset X$ are convex. The restriction $F_{|[x]}:[x] \To H$ 
is an isometric embedding by definition of $\tilde d$.
\item For any $x\in X$ the subsets $[x],Z_x\subset X$ are Alexandrov 
spaces with curvature $\geq\kappa$.
\end{enumerate}
\end{lemma}
\begin{proof} $(i)$ For a unit-speed geodesic $\gamma$ in $X$ set $\bar{A}=(F\circ \gamma)' \in H$ and
 $A= \sqrt{1-\bar{A}^2}$. Then for all $s,t$ we have $\tilde d([\gamma(s)],[\gamma(t)])=A |s-t|$, i.e. $\gamma$ 
is a geodesic with velocity $A$ with respect to $\tilde d$ proving the claim.

$(ii)$ Convexity follows, since $F$ and $\pi$ are affine. The second claim follows
immediately from the definition of $\tilde d$.

$(iii)$ By continuity of $F$ both subsets $[x],Z_x \subset X$ are closed. As closed, convex subsets of the Alexandrov space $X$ they are themselves Alexandrov spaces with the same curvature bound as $X$.
 
\end{proof}

Let $o \in X$ be a point. Due to Lemma \ref{lem:evaluation} the evaluation map $E_o$ is affine and normalized. 
We consider the special case of $H=\Aff^*$ and $F=E_o: X \To \Aff^*$. Since the value of $||E_o(x)-E_o(y)||^2$ 
does not depend on the point $o$, neither does the pseudometric $\tilde d$ given by Proposition \ref{prp:pseudometric}. 
Above we have seen that the space $Y$ is geodesic and that the embedding $i:X \To Y \times \Aff^*$, $x\MTo ([x],E_o(x))$ is isometric.
We are going to show that the data $(Y,i)$ satisfy all properties listed in Theorem \ref{thm:char_affine}.

Property $(a)$, $(b)$ and the existence part of property $(c)$ follow as in \cite[p.~12]{LS07}. For convenience we briefly recall the argument. Property $(a)$ holds by definition of $Y$. Let $f \in \tilde \Aff (X)$ be a Lipschitz continuous affine function on $X$. Define $\hat f: \Aff^* \To \R$ by 
$\hat f(\xi)=\xi([f])+f(o)$, where $[f]$ is the class of $f$ in $\Aff(X)$. Then $\hat f$ is an affine function 
on $\Aff^*$ and
\[
		\hat f(E_o(x))=E_o(x)([f])+f(o)=f(x)-f(o)+f(o)=f(x)
\]
and hence $\hat f \circ p_{\Aff^*} \circ i=f$. Applying this to a lift of a Lipschitz continuous affine function on $Y$ to $X$ shows that every Lipschitz continuous affine on $Y$ is constant. Hence, property $(b)$ holds. 
Each isometry $g$ of $X$ sends affine functions to affine functions and preserves the Lipschitz constant. Hence it induces an isometry of $\Aff^*$ and of $Y$. By construction the induced isometry of $Y \times \Aff^*$ is an extension of $g$, that is the existence part of property $(c)$ holds, too.

In view of proving Theorem \ref{thm:char_affine} it remains to establish the statement on the dimension of $\Aff^*$, the openness statement in $(d)$ and the uniqueness statements. This will be achieved in the subsequent sections.

\subsection{Local product structure} 
The tangent cones of $[x],Z_x \subset X$ at a point $x\in X$ can be regarded as subsets of $T_x X$. Recall from Lemma \ref{lem:normalized_dense} and Definition \ref{dfn:regular_normalized} that for a point $x$ in the interior of $X$ there is a splitting $T_x = C_x' \times H_x$ with $H_x$ being isometric to $\Aff^*$ and the differential $d_x E_o$ 
being given by the projection to the second factor composed with an isometry.

\begin{lemma} \label{lem:alex_dim} For a point $x$ in the interior of $X$ we have
\begin{enumerate}
\item $T_x Z_x =C_x'$. In particular, $\dim Z_x = \dim X - \dim \Aff$.
\item $T_x [x] = H_x$. In particular, $\dim [x] = \dim \Aff$.
\end{enumerate}
\end{lemma}
\begin{proof} Note that if $\gamma$ is a curve through $x$ in one of the convex subsets $Z_x$ or $[x]$, 
then it has a well-defined tangent vector $\gamma^+$ at $x$ as a curve in $X$ if and only if it has a 
well-defined tangent vector $\gamma^+$ at $x$ as a curve in $Z_x$ or $[x]$ and that in this case the 
two tangent vectors coincide (cf. \cite[2.1]{Pe07}). We represent directions at $x$ by quasigeodesics 
$\gamma$ starting at $x$. 

$(i)$ Suppose $\gamma$ is a quasigeodesic in $Z_x$. Then its tangent vector at $x$ is mapped to $0$ 
by $d_x E_o$ and hence lies in $C'_x$. If $\gamma$ is a quasigeodesic in $X$ representing a direction 
in $C'_x$, then it is mapped to a point by $E_o$, since $E_o$ is affine. This means that $\gamma$ lies 
in $Z_x$ and is thus tangent to $Z_x$.

$(ii)$ Since $E_o$ embeds the class $[x]$ isometrically into $\Aff^*$, any vector tangent to $[x]$ at $x$ belongs to $H_x$. Conversely, suppose that 
$\gamma: [0, \epsilon) \To \interior X$ is a quasigeodesic with $\gamma(0)=x$ and $\gamma^+(0) \in H_x$. 
For small $s,t$ it follows that 
\[
	|| E_o(\gamma(s))-E_o(\gamma(t))|| = || d_x E_0 (\gamma^+(0))|| \cdot |s-t|= |s-t| \geq d(\gamma(s),\gamma(t))
\]	
and hence $|| E_o(\gamma(s))-E_o(\gamma(t))||= d(\gamma(s),\gamma(t))$, since $E_o$ is $1$-Lipschitz.
 We see that $\gamma$ is a geodesic that stays in $[x]$ and that is thus tangent to $[x]$ at $x$.
\end{proof}

We obtain the following corollaries.

\begin{corollary}\label{cor:open_iso_emb} Let $x \in X$ and $r>0$ be such that $B_r(x)$ is contained
 in the interior of $X$. Then ${E_o}_{|[x] \cap B_r(x)}:[x] \cap B_r(x) \To B_r(F(x))\subset H$ is an
 isometry. 
\end{corollary}
\begin{proof} 
We already know that the map in question is an isometric embedding. Due to Lemma \ref{lem:alex_dim}, $(ii)$, for a point $z\in B_r(F(x))-\{F(x)\}$ there exists a quasigeodesic $\gamma$ starting at $x$ tangent to $[x]$ such that $d_x E_o (\gamma^+(0))$ is parallel to $z-F(x)$. By assumption the restriction of $\gamma$ to $[0,T]$, $T=||z-F(x)||$, is contained in the interior of $X$. Hence, the proof of Lemma \ref{lem:alex_dim}, $(ii)$, moreover shows that the restriction of $\gamma$ to $[0,T]$ is contained in $[x] \cap B_r(x)$. Since  ${E_o}_{|[x] \cap B_r(x)} \circ \gamma$ is a line, we have $E_o(\gamma(T))=z$ and the claim follows.
\end{proof}

\begin{corollary}\label{cor:boundary_inclusion} Let $x$ be a point in the interior of $X$. Then
 the boundary of $[x]$ is contained in the boundary of $X$.
\end{corollary}

In the sequel we show that a small neighborhood $U$ of a regular point $x$ in $X$ has a 
product structure with factors $Z_x \cap U$ and $[y] \cap U$, $y \in Z_x \cap U$. To 
this end, we need the first inclusion of the following continuity statement. The other inclusion is needed later in form of Corollary \ref{cor:openextension} (cf. Lemma \ref{lem:pi_open}).

\begin{lemma}\label{lem:Hausdorffcont}
Let $x_i\in X$ be a sequence of points converging to a point $x$ in the interior of $X$. 
Then for every radius $r>0$ and $\epsilon>0$
there exists $M\in\N$ such that 
$$
|[x]\cap \overline{B_r(x)},[x_i]\cap \overline{B_r(x)}|_H<\epsilon\text{ for all } i \geq M.
$$
Here $|\cdot|_H$ denotes the Hausdorff distance, cf. \cite{MR1835418}.
\end{lemma}

\begin{proof}
Assume that there are points $y_i\in[x_i]\cap \overline{B_r(x)}$ such that $d(y_i,[x]\cap \overline{B_r(x)})\geq \epsilon$ 
for all $i$. Without loss of generality $y_i\to y$ for some point $y\in \overline{B_r(x)}$. 
Then $d(y,[x]\cap \overline{B_r(x)})\geq \epsilon$. On the other hand $\|F(x)-F(y)\|=\lim_{i\to \infty}\|F(x_i)-F(y_i)\|=\lim_{i\to \infty}d(x_i,y_i)=d(x,y)$. 
Hence $y\in[x]$, a contradiction. Therefore, for $i$ large, $[x_i]\cap \overline{B_r(x)}$ is 
contained in an $\epsilon$-tubular neighborhood of $[x]\cap \overline{B_r(x)}$.

Now assume that for every $i$ there is a point $z_i\in[x]\cap\overline{B_r(x)}$ with
 $d(z_i,[x_i]\cap\overline{B_r(x)})\geq\epsilon$. Then $(z_i)$ subconverges to a point 
$z\in[x]\cap\overline{B_r(x)}$ with $d(z,[x_i]\cap\overline{B_r(x)})\geq\epsilon/2$ for $i$ large. In particular, we have $z\neq x$. In the following we show that $z$ is a limit of points in $[x_i]$ which is a contradiction and thus will prove the lemma.

Since the interior of $X$ 
is open in $X$ (cf. Section \ref{sub:tangent_cones}), there exists some $\rho>0$ such that $B_{3\rho}(x)$ lies in the interior of $X$. 
Thus for large $i$ also $B_{2\rho}(x_i)$ lies in the interior of $X$. By Corollary \ref{cor:open_iso_emb} 
this implies that $[x_i]\cap B_{2\rho}(x_i)$ is isometric to an $m$-dimensional Euclidean ball of 
radius $2\rho$, where $m=\dim \Aff$. By passing to a subsequence we may assume that $[x_i]\cap \overline{B_{\rho}(x_i)}$ 
converges to $[x]\cap \overline{B_{\rho}(x)}$ in Hausdorff topology. Let $[xz]\subset [x]$ be the minimizing geodesic between $x$ and $z$ (cf. Lemma \ref{lem:prop1}, $(ii)$). 
Choose a point $z'$ on $(xz)\cap [x]\cap \overline{B_{\rho}(x)}$ and lift it to points $z'_i\in [x_i]\cap \overline{B_{\rho}(x_i)}$. 
Assume that there exists $\tau>0$ such that the maximal geodesic extension $[x_iy_i]$ of $[x_iz'_i]$ has length
$\leq d(x,z)-\tau$. Then, by Lemma \ref{lem:prop1}, $(ii)$, the point $y_i$ has to lie in the boundary of $[x_i]$ which by Corollary \ref{cor:boundary_inclusion} 
is contained in the boundary of $X$. Since the sequence $y_i\in \partial X$ is bounded we can assume that it converges to a point $y\in \partial X$. Moreover, by the
triangle inequality and the fact that geodesics in $X$ do not branch, we have $y\in \partial X \cap(xz)$. But this contradicts the fact that geodesics starting in the interior of $X$ do not pass boundary points in their interior (cf. Section \ref{sub:tangent_cones}). Therefore there exist geodesic extensions $[x_iz_i]$ of $[x_iz'_i]$ whose lengths converge to $d(x,z)$. Again, since geodesics in $X$ do not branch, the sequence $z_i\in [x_i]$ converges to $z$ and thus the lemma is proven.
\end{proof}

\begin{corollary}\label{cor:openextension}
 Let $U\subset X$ be an open subset and set $\hat U :=\pi^{-1}(\pi(U))$. Then the subset $\hat U\cap \interior X$ is open.
\end{corollary}

\begin{proof}
 Let $x$ be a point in $\hat U\cap \interior X$ and $(x_i)$ a sequence with $x_i\to x$. Choose a point $x'\in[x]\cap U$. 
 By Lemma \ref{lem:Hausdorffcont},
 there is a sequence $(x'_i)$ converging to $x'$ and such that $x'_i\in[x_i]$ for large $i$. Therefore, for large $i$
 we have $x_i \in \hat U\cap \interior X$ and thus $\hat U\cap \interior X$ is open. 
\end{proof}

Next we establish the local product structure around regular points.

\begin{lemma} \label{lem:fiber_structure}
 Let $x$ be a regular point in $X$. Then there exists some $r>0$ such that for any $y\in B_r(x)$ there exists some $z \in Z_x$
 such that  $[y]\cap Z_x= \{z\}$.
\end{lemma}
\begin{proof} Since $x$ is regular in $X$ a small neighborhood $U$ of $x$ in $X$ is homeomorphic to 
$\R^n$. By Lemma \ref{lem:alex_dim} the point $x$ is also regular in $Z_x\cap U$. Let $D^k$ be a small 
open disk in $Z_x$ around $x$ whose closure is a closed disk contained in $Z_x \cap U$. The inclusion $\varphi: \overline{D}^k \To U$ is a topological embedding. By Corollary \ref{cor:open_iso_emb}
 we can choose $U$ and $D^k$ in such a way that  $\bigcap_{y\in \overline{D}^k}F(U\cap[y])$
 contains a closed disk $\overline{D}^m$ around $F(x)$ in $H$ where $m=\dim H =n-k$. Set 
$\psi_y = ((F_{|[y]})^{-1})_{|\overline{D}^m} : \overline{D}^m \To U$. By Lemma \ref{lem:Hausdorffcont}
 $D^k \ni y\mapsto \mathrm{im}(\psi_y)$ is continuous with respect to the Hausdorff metric. Moreover, 
for all $y \in D^k$ we have $F \circ \psi_y=\mathrm{id}_{\overline{D}^m}$. Consider the map
\[
			\phi: D^k \times D^m \To U, \; (y,z) \mapsto \psi_y(z).
\]
By shrinking $D^k$ and $D^m$ if necessary, we can assume that the image of $\phi$ is contained in a compact neighborhood $K$ of
 $x$ with $K\subset U$. We claim that $\phi$ is continuous. Indeed, suppose $(y_i,z_i)$ is a sequence 
in $D^k \times D^m$ converging to $(y,z) \in D^k \times D^m$ such that $\phi(y_i,z_i)$ does not 
converge to $\phi(y,z)$. We can assume that $\phi(y_i,z_i)$ converges to a point $p \in U$ distinct
 from $\phi((y,z))$. By continuity of $D^k \ni y\mapsto \mathrm{im}(\psi_y)$ we have $p \in \mathrm{im}(\psi_y)$.
 By continuity of $F$ we have
\[
	F(p)=\lim_{i \rightarrow \infty} F(\phi(y_i,z_i))= \lim_{i \rightarrow \infty} F(\psi_{y_i}(z_i)) = \lim_{i \rightarrow \infty} z_i = z = F(\phi(y,z)).
\]
This contradicts the injectivity of $F$ on $[y]$. Since the map $\phi$ is also injective, its image 
is open by the domain invariance theorem. The claim follows.
\end{proof}

\begin{lemma}\label{lem:Zclosed}
 For any $x\in X$ the image of $Z_x$ under $\pi:X\to Y$ is closed.
\end{lemma}

\begin{proof}
 The restriction $\pi|_{Z_x}$ is an isometric embedding. The claim follows since $Z_x\subset X$ is closed. 
\end{proof}

\begin{lemma} \label{lem:pi_open}
 Let $x$ be a regular point in $X$. Then there exists some $r>0$ such that the map $\pi$ restricts 
to an isometry $B_r(x)\cap Z_x\to B_r(\pi(x)) \subset Y$.
\end{lemma}

\begin{proof}
 Clearly the restriction of $\pi$ is isometric. By Lemma \ref{lem:fiber_structure} there is a radius $r>0$ such that for any $y\in B_r(x)$ we have 
$[y]\cap Z_x\neq\emptyset$. Set $B:=B_r(x)$ and define $\hat B$ as in Corollary \ref{cor:openextension}.
 Assume that there exists a point $p$ in $X-\hat B$
 such that $d(\pi(x),\pi(p))<r$. Choose $p'\in[xp]$ maximal such that $\pi([xp'))\subset\pi(Z_x)$. 
By Lemma \ref{lem:fiber_structure}, $p'\neq x$. By Lemma \ref{lem:Zclosed}, $\pi(p')$ is contained in
 $\pi(Z_x)$. Set $q':=\pi|_{Z_x}^{-1}(\pi(p'))$. Then $[q']=[p']$ and $d(x,q')=d(\pi(x),\pi(p'))\leq d(\pi(x),\pi(p))<r$, 
since $\pi$ is affine. It follows that $q'\in B$ and $p' \in \hat B$. In particular, $p' \neq p$. 
From the maximality, we conclude that $p'\in\partial X$, since $\hat B\cap\interior X$ is open by Corollary \ref{cor:openextension}.
This is a contradiction, since $x$ lies in the interior of $X$ and $p'$ lies in the interior of $[xp]$ (cf. Section \ref{sub:tangent_cones}).
Therefore, $\pi$ restricts to an isometry $B_r(x)\cap Z_x\to B_r(\pi(x))$. 
\end{proof}

The following corollary proves property $(d)$ in Theorem \ref{thm:char_affine}.

\begin{corollary} \label{cor:i_open}
 Let $x$ be a regular point in $X$. Then there exists some $r>0$ such that the embedding 
$i: X \To Y \times \R^m$ restricts to an isometry $B_r(x)\to B_r(i(x)) \subset Y \times \R^m$.
\end{corollary}
\begin{proof}
Choose $r>0$ small enough such that the conclusion of Lemma \ref{lem:pi_open} holds, such that
$B_r(x)$ is homeomorphic to $\R^n$ and such that $B_r(x)\cap Z_x$ is homeomorphic to $\R^{n-m}$. Then, from Lemma \ref{lem:pi_open}
we know that $\pi(B_r(x)\cap Z_x)$ is also homeomorphic to $\R^{n-m}$. Perhaps after decreasing $r$, the claim now follows from the invariance of domain theorem.
\end{proof}

\subsection{Curvature bound}

For $\kappa \geq 0$ the curvature bound for $Y$ can be easily established similarly as in
 \cite[5.3]{LS07} in the nonpositive curvature case. For general $\kappa$ we argue as follows.
 The set of regular points in $X$ is convex \cite[Corollary~1.10]{Pe98} and dense in 
$X$ \cite[Corollary~10.9.13]{MR1835418}. Let $Y_0\subset Y$ be its image under $\pi$. By continuity of $\pi$ this subset $Y_0$ is dense in $Y$. Moreover, by Lemma \ref{lem:pi_open} 
and the existence of small convex neighborhoods in $Z_x$ \cite[3.6]{Pe93} (cf. \cite[7.1.2]{Pe07}),
 each point in $Y_0$ has a neighborhood $U$ in $Y$ that is isometric to a convex subset of some $Z_x$. 
According to \cite{Pe15} it follows that the completion of $Y_0$, and hence the completion 
of $Y$, is an $\Alex^{k}(\kappa)$ space with $k= \dim X - \dim \Aff$ by Lemma \ref{lem:pi_open} and Lemma \ref{lem:alex_dim}.

\subsection{Uniqueness statements} It follows from property $(d)$ and the fact that geodesics in the Alexandrov space $Y \times \R^m$ do not branch that the extension in Theorem \ref{thm:char_affine}, $(c)$ is unique.

Suppose an embedding $i:X \To Y \times \R^m$ with the properties described 
in Theorem \ref{thm:char_affine} is given. We identify $\R^m$ with its dual via the metric $(-,-)$ of $\R^n$. This 
yields a homomorphism $\theta: \R^m\To \Aff(X)$ which is surjective by property $(b)$ and 
injective by property $(d)$. In fact, it is an isometry by property $(d)$ and the definition of the norm of $\Aff(X)$ as the optimal Lipschitz constant. In particular, the dimension $m$ is an isometry invariant of $X$. Set
 $F=p_H \circ i : X \To \R^m$. After composition with a translation of $\R^m$ we can assume that any point, say $o \in X$, is mapped to $0 \in \R^m$ by $F$. We claim that $F=\theta^{*}\circ E_o$. Indeed, 
let $x\in X$, let $f: X \To \R$ be a Lipschitz continuous affine function and let $\hat f: \R^m \To \R$ 
be given as in property $(b)$. Then
\[
		\left( F(x), \theta^{-1} ([f])\right)  = \hat f (F(x)) -\hat f(0)= f(x)-f(o)=E_o(x)([f])=\left(\left(\theta^* \circ E_o\right)(x),\theta^{-1} ([f]) \right)
\]
and so the claim follows since our choices were arbitrary. For $\pi=p_Y \circ i : X \To Y$ the metric on $Y$ must satisfy
\[d(\pi(y),\pi(z))= \sqrt{d(y,z)^2-||E_o(y)-E_o(z)||^2}.\]
By property $(a)$ the metric on $Y$ is completely determined by this identity. More precisely, by Proposition \ref{prp:pseudometric}
$\sqrt{d(y,z)^2-||E_o(y)-E_o(z)||^2} $ defines a pseudometric on $X$ with $Y$ being the induced metric space. Since this pseudometric on $X$ does not depend on the choice of $o$, we see that the space $Y$ is determined by $X$ up to isometry.

If $\tilde i: X \To \tilde Y \times \R^{\tilde m}$ is another embedding with the properties stated in Theorem \ref{thm:char_affine}, then we have seen that $m=\tilde m$ and $Y=\tilde Y$ up to isometry. Moreover, the maps $F=p_H \circ i$ and $\tilde F=p_H \circ \tilde i$ coincide after composition with a translation of $\R^m$. Therefore, also the maps $\pi=p_Y \circ i$ and $\pi=p_Y \circ \tilde i$ coincide as they are defined by the corresponding pseudometrics which in turn only depend on $F$. But this means that the embedding $i:X \To Y \times \R^m$ is uniquely determined up to composition with isometries as claimed.
\newline

\emph{Acknowledgements.} We would like to thank Vitali Kapovitch, Alexander Lytchak and Anton Petrunin for comments and discussions on different aspects of Alexandrov geometry. We would also like to thank the anonymous referee for his remarks that helped to improve the exposition. The first named author was partly supported by a `Kurzzeitstipendium für Doktoranden' by the German Academic Exchange Service (DAAD).

\end{document}